\documentclass[11pt,twoside]{amsart}

\numberwithin{equation}{section}
\usepackage{latexsym,amssymb,amsmath,amsthm,amsfonts,amscd}
\usepackage{enumerate}
\usepackage[table]{xcolor}
\usepackage{graphicx}
\usepackage{tikz}
\usetikzlibrary{matrix,arrows}
\usepackage{multicol}
\usepackage{multirow}

\definecolor{VerdeOlivo}{rgb}{0.3,0.5,0.1}
\definecolor{Magenta}{rgb}{.65,0.15,.2}
\definecolor{Gris}{gray}{0.3}

\newtheorem{Theorem}{Theorem}[section] 
\newtheorem{Definition}[Theorem]{Definition}
\newtheorem{Proposition}[Theorem]{Proposition}  
\newtheorem{Lemma}[Theorem]{Lemma} 
\newtheorem{Corollary}[Theorem]{Corollary}

\newtheorem{Example}[Theorem]{Example}

\newtheorem*{Theorem2.2}{Theorem 2.2}
\newtheorem*{Corollary2.6}{Corollary 2.6}
\newtheorem*{Proposition3.1}{Proposition 3.1}
\newtheorem*{Corollary3.2}{Corollary 3.2}
\theoremstyle{definition}
\newcommand{\bff}[1]{{\bf #1}}

\begin{document}

\title[Arithmetical structures on graphs with connectivity one]{Arithmetical structures on graphs with connectivity one}

\author{Hugo Corrales}
\email[H.~Corrales]{hhcorrales@gmail.com}
\author{Carlos E. Valencia}
\email[C.~Valencia\footnote{Corresponding author}]{cvalencia@math.cinvestav.edu.mx, cvalencia75@gmail.com}

\thanks{The authors were partially supported by SNI}
\address{
Departamento de
Matem\'aticas\\
Centro de Investigaci\'on y de Estudios Avanzados del IPN\\
Apartado Postal 14--740 \\
07000 Mexico City, D.F. 
} 


\begin{abstract}
Given a graph $G$, an arithmetical structure on $G$ is a pair of positive integer vectors $(\bff{d},\bff{r})$ 
such that $\mathrm{gcd}(\bff{r}_v\, | \,v\in V(G))=1$ and
\[
(\mathrm{diag}(\bff{d})-A)\bff{r}=0,
\]
where $A$ is the adjacency matrix of $G$. 
We describe the arithmetical structures on graph $G$ with a cut vertex $v$ in terms of the arithmetical structures on their blocks.
More precisely, if $G_1,\ldots,G_s$ are the induced subgraphs of $G$ obtained from each of the connected components of $G-v$ by adding the vertex $v$
and their incident edges, then the arithmetical structures on $G$ are in one to one correspondence with the $v$-rational arithmetical structures on the $G_i$'s.
We introduce the concept of rational arithmetical structure, which corresponds to an arithmetical structure where some of the integrality conditions are relaxed.
\end{abstract}

\maketitle

\section{Introduction}
Given a multidigraph $G=(V,E)$ and an integral positive vector $\bff{d}$, let
\[
L(G,\bff{d})_{u,v}=
\begin{cases}
\bff{d}_u&\textrm{if }u=v,\\
-m_{u,v}&\textrm{if }u\neq v,
\end{cases}
\]
where $m_{u,v}$ is the number of the directed arcs between $u$ and $v$.
If $\bff{d}$ is the degree vector of $G$, then $L(G,\bff{d})$ is the \emph{Laplacian matrix} of $G$,
which is a matrix with rank $|V|-1$ and nullspace spanned by the all-ones vector $\bff{1}$.
An \emph{arithmetical structure} on $G$ is a pair $({\bf d},{\bf r})\in \mathbb{N}_+^{|V|}\times \mathbb{N}_+^{|V|}$ such that
\[
L(G,{\bf d}){\bf r}^t=\bff{0}^t\qquad\textrm{and}\qquad \textrm{gcd}({\bf r}_v)_{v\in V}=1.
\]
Thus, the notion of arithmetical structures on $G$ generalizes the Laplacian matrix of $G$.
The arithmetical structures on simple graphs were introduced by Lorenzini in~\cite{Lorenzini89} as some intersection matrices that arise in 
the study of degenerating curves in algebraic geometry.
More precisely, the vertices of $G$ represent the components of a degeneration of a given curve, 
the edges represent the intersections of the components, and the entries of $\bff{d}$ are their self-intersection numbers.
Lorenzini proved in~\cite{Lorenzini89} that if $G$ is a simple connected graph, then there are a finite number of arithmetical structures.
In~\cite{arithmetical} this result was generalized to strongly connected multidigraphs.
On the other hand, given an arithmetical structure $(\bff{d},\bff{r})$ on $G$, let
\[
K(G,\bff{d},\bff{r})=\mathrm{ker}(\bff{r}^t)/\mathrm{Im} \, L(G,\bff{d})^t
\]
be its critical group, which generalizes the concept of the critical group of $G$ introduced in~\cite{biggs97}. 
Moreover, recently in~\cite{klivans} there was defined a sandpile group for the matrices $L(G,\bff{d})$ obtained from arithmetical structures. 
The Laplacian matrix of a graph is very important in spectral graph theory and in general in algebraic graph theory, 
see for instance~\cite{godsil} and the references therein.

Since the number of arithmetical structures on any strongly connected multidigraph is finite, it seems possible to describe them.
More precisely, given a connected graph $G$, we are interested in describing the set
\[
\mathcal{A}(G)=\left\{({\bf d},{\bf r})\in\mathbb{N}_+^{|V|}\times\mathbb{N}_+^{|V|} \,\big| \,(\bff{d},\bff{r})\textrm{ is an arithmetical structure on } G\right\}.\\
\]
In~\cite{counting} there was described in detail the combinatorics of the arithmetical structures on the path and cycle graphs.
Also, in~\cite{arithmetical} there was introduced and studied the arithmetical structures in the general setting of $M$-matrices 
and in particular the arithmetical structures on complete graphs, paths, and cycles were studied.
Moreover, there was described a subset of the arithmetical structures on the 
clique-star transformation of a graph $G$ in dependence on the arithmetical structures on $G$.
Apart from the arithmetical structures on the path and cycle, in general the description of the arithmetical structures on a graph is a very difficult problem.
For instance, for the complete and star graph, its arithmetical structures are in one to one correspondence with a variant of the Egyptian fractions.
Therefore is important to have an idea of the complexity of the arithmetical structures on a graph with a cut vertex.

In this article we study the arithmetical structures on a connected graph with a cut vertex.
The main results of this article are contained in Section~\ref{connec1}, where we establish the relation between the 
arithmetical structures on a graph with one cut vertex and the arithmetical structures on their blocks.
More precisely, given connected graphs $G_1$ and $G_2$ and a vertex $v_i$ of $G_i$,
let $G_1\vee_v G_2$ be the graph that results if $v_1$ and $v_2$ are identified to a new vertex $v$.
Moreover, let $\mathcal{A}_{v_i}(G_i)$ be the arithmetical structures on $G_i$ in which the integrality condition over ${\bf d}_{v_i}$ has been relaxed.
Given $({\bf a},{\bf r}_1)\in \mathcal{A}_{v_1}(G_1)$ and $({\bf b},{\bf r}_2 )\in \mathcal{A}_{v_2}(G_2)$, let 
\[
l(\bff{r}_1,\bff{r}_2)=
\begin{cases}
\frac{l}{({\bf r}_1)_{v_1}}\cdot ({\bf r}_1)_u & \text{ if } u\in V(G_1),\\
\frac{l}{({\bf r}_2)_{v_2}}\cdot({\bf r}_2)_u & \text{ if } u\in V(G_2- v_2),
\end{cases}
\]
where $l=\mathrm{lcm}((\bff{r}_1)_{v_1},(\bff{r}_2)_{v_2})$.
The main result of this article is the following description of the arithmetical structures on $G$ in terms  of the arithmetical structures on $G_1$ and $G_2$.
\begin{Theorem2.2}
If $G_1$ and $G_2$ are connected graphs and $v_i$ is a vertex of $G_i$, then
\[
\mathcal{A}(G_1\vee_v G_2)=\{ \left({\bf a}+_v{\bf b},l(\bff{r}_1,\bff{r}_2)\right) \, | \, {\bf a}+_v{\bf b}\in \mathbb{N}_+^{|V(G)|} \text{ and } 
({\bf a},{\bf r}_1)\in \mathcal{A}_{v_1}(G_1), ({\bf b},{\bf r}_2 )\in \mathcal{A}_{v_2}(G_2)\},
\]
where
\[
({\bf a}+_v{\bf b})_u=
\begin{cases}
{\bf a}_u+{\bf b}_u & \text{ if } u=v,\\
{\bf a}_u & \text{ if } u\in V(G_1- v),\\
{\bf b}_u & \text{ if } u\in V(G_2- v).
\end{cases}
\]
\end{Theorem2.2}

A typical case of a graph with cut vertices is a tree.
In Section~\ref{ontrees} we apply our result to the problem of describing the arithmetical structures on a tree.
We give a description of the arithmetical structures on a star graph and give a partial description of the arithmetical structures on a general tree in terms of its star graphs.

In general we give the following description of the arithmetical structures on a tree.
Given a tree $T$, let $\overset{\rightarrow}{T}$ be the digraph obtained from $T$ by replacing each edge $uv$ of $T$ by the arcs $uv$ and $vu$.
Let $a: E(\overset{\rightarrow}{T})\rightarrow \mathbb{Q}_+$ be a weight function on the arcs of $\overset{\rightarrow}{T}$.

\begin{Proposition3.1}
If $T$ is a tree, then
\[
\mathcal{A}(T)=\left\{\big({\bf d},{\bf r}\big)\in \mathbb{N}_+^{|V(T)|}\times \mathbb{N}_+^{|V(T)|}\, \left|\, \sum_{uv\in E(T)} \bff{a}_{u,v}=\bff{d}_v \,\, \forall \, v\in V(T) \text{ and }
\bff{a}_{u,v}\bff{a}_{v,u}= 1 \, \,\forall \,  uv\in E(T)\right.\right\}. 
\]
\end{Proposition3.1}

Finally, in the particular case that $T$ is a star graph, we get the following description.

\begin{Corollary3.2}
If $S_m$ is the star graph with center at $v$ and with $1,\ldots,m$ its leaves, then
\[
\mathcal{A}(S_m)=\left\{\left({\bf d},{\bf r}\right) \in \mathbb{N}_+^{m+1}\times \mathbb{N}_+^{m+1}\, \left| \, \bff{d}_v=
\sum_{i=1}^m \frac{1}{\bff{d}_i},\ \bff{r}_v=\mathrm{lcm}\{\bff{d}_i\}_{i=1}^m\textrm{ and }{\bf r}_i=\frac{{\bf r}_v}{\bff{d}_i}\right.\right\}.
\]
\end{Corollary3.2}

\section{Arithmetical structures on graphs with connectivity one}\label{connec1}
Given two connected graphs $G_1,G_2$ and $v_i$ a vertex of $G_i$, let 
\[
G=G_1\vee_v G_2
\] 
be the graph obtained from $G_1$ and $G_2$ by identifying the vertices $v_1$ and $v_2$ to a new vertex, denoted by $v$.
That is, $v$ is a cut vertex of $G$. 
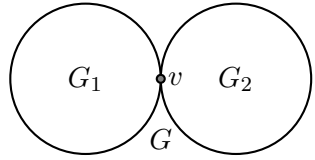
\begin{figure}[h]
\begin{tikzpicture}[scale=1,line width=0.8pt]
\tikzstyle{every node}=[minimum width=3pt, inner sep=0pt, circle]
\draw (0,0) circle (10mm);
\draw (2,0) circle (10mm);
\draw (0:1) node (v1) [draw,fill=gray] {};
\draw (1.2,-0) node () {$v$};
\draw (0,0) node () {$G_1$};
\draw (2,0) node () {$G_2$};
\draw (1,-0.8) node () {$G$};
\end{tikzpicture}
\caption{A graph with a cut vertex.}
\label{figura2}
\end{figure}

\begin{Lemma}\label{det}
Let $G_1$ and $G_2$ be two connected graphs and $v_i$ a vertex of $G_i$.
If $G$ is equal to $G_1\vee_v G_2$, then
\[
{\rm det}(G,X)|_{x_v=x_{v_1}+x_{v_2}}={\rm det}(G_1 \!-\! v_1,X)\cdot {\rm det}(G_2,X)+{\rm det}(G_2 \!-\! v_2,X)\cdot {\rm det}(G_1,X),
\]
where ${\rm det}(G,X)$ is the determinant of the generalized Laplacian matrix
\[
L(G,X)_{u,v}=\begin{cases}
x_u&\textrm{if }u=v,\\
-m_{u,v}&\textrm{if }u\neq v,
\end{cases}
\]
with $x_u$ an indeterminate indexed by the vertex $u$.
\end{Lemma}
\begin{proof}
Let $D$ be the digraph obtained from $G$ by replacing each edge $uv$ of $G$ by the arcs $uv$ and $vu$.
Also, given $U\subseteq V(D)$, let $D[U]$ be the induced subgraph of $D$ by $U$ and let $D[U]^{wl}$ be the graph obtained from $D[U]$ when we delete any possible loops. 

By~\cite[Theorem 4.1]{critical} we have that
\begin{eqnarray*}
\mathrm{det}(L(G,X))&=&\mathrm{det}(L(D,X))=\sum_{U\subseteq V(D)} \mathrm{det}(-A(D[U]^{wl})) \prod_{u\notin U} x_u\\
&=& \sum_{v\notin U\subseteq V(D)} \mathrm{det}(-A(D[U]^{wl})) \prod_{u\notin U} x_u+
\sum_{v\in U\subseteq V(D)} \mathrm{det}(-A(D[U]^{wl})) \prod_{u\notin U} x_u\\
&=&x_v \sum_{U\subseteq V(D- v)} \mathrm{det}(-A(D[U]^{wl})) \prod_{v\neq u\notin U} x_u+ 
\sum_{v\in U\subseteq V(D)} \mathrm{det}(-A(D[U]^{wl})) \prod_{u\notin U} x_u.
\end{eqnarray*}

Now, let $\overset{\rightarrow}{\mathcal{F}}$ be the set of spanning directed $1$-factors of $D$ 
(a directed $1$-factor of $D$ is a subdigraph $C$ such that $d_C^+(v)=d_C^-(v)=1$ for all $v\in V(C)$)
and let $c(C)$ be the number of connected components of $C$.
Since $\mathrm{det}(-A(D))=\sum_{C\in \overset{\rightarrow}{\mathcal{F}}} (-1)^{c(C)}$ 
and the connected directed $1$-factors of $D$ can not contain vertices of $G_1- v_1$ and $G_2- v_2$ (because $v$ is a cut vertex of $D$), then
\begin{eqnarray*}
\mathrm{det}(L(G,X))|_{x_v=x_{v_1}+x_{v_2}}&=&
(x_{v_1}+x_{v_2}) \mathrm{det (L(G_1\!-\! v_1,X))} \mathrm{det (L(G_2\!-\! v_2,X))}\\
&+&\mathrm{det (L(G_1\!-\! v_1,X))}\left[ \sum_{v\in U\subseteq V(G_2)} \mathrm{det}(-A(D[U]^{wl})) \prod_{u\notin U} x_u\right]\\
&+& \mathrm{det (L(G_2\!-\! v_2,X))}\left[ \sum_{v\in U\subseteq V(G_1)} \mathrm{det}(-A(D[U]^{wl})) \prod_{u\notin U} x_u\right]\\
&=& \mathrm{det (L(G_1\!-\! v_1,X))} \left[ x_{v_2} \mathrm{det (L(G_2\!-\! v_2,X))}+\!\!\!\!\!\!\!\!\! 
\sum_{v\in U\subseteq V(G_2)} \!\!\!\!\!\!\!\!\! \mathrm{det}(-A(D[U]^{wl})) \prod_{u\notin U} x_u\right]\\
&+& \mathrm{det (L(G_1\!-\! v_1,X))} \left[ x_{v_1}  \mathrm{det (L(G_1\!-\! v_1,X))}+
\!\!\!\!\!\!\!\!\! \sum_{v\in U\subseteq V(G_1)} \!\!\!\!\!\!\!\!\! \mathrm{det}(-A(D[U]^{wl})) \prod_{u\notin U} x_u\right]\\
&=& {\rm det}(G_1\!-\! v_1,X) {\rm det}(G_2,X)+{\rm det}(G_2\!-\! v_2,X) {\rm det}(G_1,X).
\end{eqnarray*}
\end{proof}

Before our principal result is presented, we need to introduce some notation.
Given ${\bf a}\in \mathbb{R}^{V(G_1)}$ and ${\bf b}\in \mathbb{R}^{V(G_2)}$, let
\[
({\bf a}+_v{\bf b})_u=
\begin{cases}
{\bf a}_u+{\bf b}_u & \text{ if } u=v,\\
{\bf a}_u & \text{ if } u\in V(G_1- v),\\
{\bf b}_u & \text{ if } u\in V(G_2- v).
\end{cases}
\]
Moreover, given $A_1\subset \mathbb{R}^{V(G_1)}$ and $A_2\subset \mathbb{R}^{V(G_2)}$, let
\[
A_1 +_vA_2=\{ {\bf a}+_v{\bf b} \, | \, {\bf a} \in A_1 \text{ and } {\bf b} \in A_2\}.
\]
Now, given a set $U$ of vertices of $G$, we introduce the concept of a $U$-rational arithmetical structure. 
\begin{Definition}
A pair $(\bff{d},\bff{r})\in \mathbb{Q}_+^n\times \mathbb{N}_+^n$ is a {\it $U$-rational arithmetical structure} on $G$ whenever
\[
L(G,\bff{d})\bff{r}^t=\bff{0}^t
\]
and $\bff{d}_w\in\mathbb{N}_+$ for all $w\notin U$.
\end{Definition}

Let $\mathcal{A}_U(G)$ be the set of $U$-rational arithmetical structures on $G$.
That is, $\mathcal{A}_U(G)$ is the set of arithmetical structures on $G$ where the integrality condition over the vectors of $U$ is relaxed.
Clearly $\mathcal{A}(G) \subseteq \mathcal{A}_U(G)$.
Sometimes we will, for simplicity, only refer to them as rational arithmetical structures if the set of vertices of $U$ is clear.
Given $({\bf a},{\bf r}_1)\in \mathcal{A}_{v_1}(G_1)$ and $({\bf b},{\bf r}_2 )\in \mathcal{A}_{v_2}(G_2)$, let $l=\mathrm{lcm}((\bff{r}_1)_{v_1},(\bff{r}_2)_{v_2})$ and
\[
l(\bff{r}_1,\bff{r}_2)=
\begin{cases}
\frac{l}{({\bf r}_1)_{v_1}}\cdot ({\bf r}_1)_u & \text{ if } u\in V(G_1),\\
\frac{l}{({\bf r}_2)_{v_2}}\cdot({\bf r}_2)_u & \text{ if } u\in V(G_2- v_2).
\end{cases}
\]

Now we are ready to state the main result of this article.

\begin{Theorem}\label{1connected}
Let $G_1,G_2$ be two connected graphs and $v_i$ a vertex of $G_i$.
If $G$ is equal to $G_1\vee_v G_2$, then
\[
\mathcal{A}(G)=\{ \left({\bf a}+_v{\bf b},l({\bf r}_1,\bff{r}_2\right) \, | \, {\bf a}+_v{\bf b}\in \mathbb{N}_+^{|V(G)|} \text{ and } ({\bf a},{\bf r}_1)\in \mathcal{A}_{v_1}(G_1), ({\bf b},{\bf r}_2 )\in \mathcal{A}_{v_2}(G_2)\}.
\]
\end{Theorem}
\begin{proof}
Let $(\bff{d}_1,\bff{r}_1)$ and $(\bff{d}_2,\bff{r}_2)$ be $v$-arithmetical structures on $G_1$ and $G_2$, respectively.
We will prove that $(\bff{d}_1+_v\bff{d}_2,l(\bff{r}_1,\bff{r}_2))$ is an arithmetical structure on $G$.
First
\begin{eqnarray*}
{\rm det}(G,\bff{d}_1+_v\bff{d}_2)
&=&{\rm det}(G_1 \!-\! v_1,\bff{d}_1')\cdot {\rm det}(G_2,\bff{d}_2)+{\rm det}(G_2 \!-\! v_2,\bff{d}_2')\cdot {\rm det}(G_1,\bff{d}_1)\\
&=&{\rm det}(G_1 \!-\! v_1,\bff{d}_1')\cdot 0+{\rm det}(G_2 \!-\! v_2,\bff{d}_2')\cdot 0=0
\end{eqnarray*}
where $\bff{d}_i'$ is the vector obtained by erasing the entry $v_i$.
It is not difficult to see that
\[
L(G, X)
=\left(\begin{array}{ccc}
L(G_1- v_1,X)& {\bf m}_1^t & {\bf 0} \\
{\bf m}_1& x_v& \bff{m}_2\\
{\bf 0} & {\bf m}_2^t&L(G_2- v_2,X)
\end{array}\right)
\]
for some vectors ${\bf m}_1,{\bf m}_2$ of integers such that 
\[
L(G_1, X)
=\left(\begin{array}{cc}
L(G_1- v_1,X)& {\bf m}_1^t\\
{\bf m}_1& x_v
\end{array}\right)
\text{ and }
L(G_2, X)
=\left(\begin{array}{cc}
x_v& \bff{m}_2\\
{\bf m}_2^t&L(G_2- v_2,X)
\end{array}\right).
\]
Thus
\[
L(G, \bff{d}_1+_v\bff{d}_2)l(\bff{r}_1,\bff{r}_2)^t
=\left(\begin{array}{ccc}
L(G_1- v_1,X)& {\bf m}_1^t & {\bf 0} \\
{\bf m}_1& x_v& \bff{m}_2\\
{\bf 0} & {\bf m}_2^t&L(G_2- v_2,X)
\end{array}\right)
l(\bff{r}_1,\bff{r}_2)^t=0.
\]

Now, for the other containment we have the following arguments.
Let $(\bff{d},\bff{r})$ be an arithmetical structure on $G$.
By~\cite[Proposition 3.3]{arithmetical}, $L(G,\bff{d})$ is an almost non-singular $M$-matrix, that is, a singular $M$-matrix with all its proper principal minors positive.
Let $\bff{d}_i'$ be the vector obtained from $\bff{d}$ that only contains the entries in the vertices of $G_i- v_i$.
Since $L(G_1- v_1,\bff{d}_1')$ and $L(G_2- v_2,\bff{d}_2')$ are submatrices of $L(G,\bff{d})$,
then $\mathrm{det}(G_i- v_i,\bff{d}_i')$ is positive.
On the other hand, let $0\leq t\leq 1$ and
\begin{equation*}
\begin{array}{ccc}
\bff{d}_1(t)_u=
\begin{cases}
{\bf d}_u&\textrm{ if } u\neq v,\\
t\bff{d}_v&\textrm{ if }u=v,
\end{cases}
&\quad\textrm{and}\quad\quad&
\bff{d}_2(t)_u=
\begin{cases}
{\bf d}_u&\textrm{ if } u\neq v,\\
(1-t)\bff{d}_v&\textrm{ if }u=v,
\end{cases}
\end{array}
\end{equation*}
vectors that depend on $t$ and such that $\bff{d}=\bff{d}_1(t)+_v\bff{d}_2(t)$.
By Lemma~\ref{det}
\[
0=\mathrm{det}(G, \bff{d})={\rm det}(G_1 \!-\! v_1,\bff{d}_1')\cdot {\rm det}(G_2,\bff{d}_2(t))
+{\rm det}(G_2 \!-\! v_2,\bff{d}_2')\cdot {\rm det}(G_1,\bff{d}_1(t))
\]
Thus
\[
\frac{{\rm det}(G_1,\bff{d}_1(t))}{{\rm det}(G_1 \!-\! v_1,\bff{d}_1')}=-\frac{{\rm det}(G_2,\bff{d}_2(t))}{{\rm det}(G_2 \!-\! v_2,\bff{d}_2')}
\]
Since $\mathrm{det}(G_i- v_i,\bff{d}_i')$ is positive, then ${\rm det}(G_1,\bff{d}_1(t)){\rm det}(G_2,\bff{d}_2(t))\leq 0$.
Since $L(G_1,\bff{d}_1(1))$ is a submatrix of $L(G,\bff{d})$, then $\mathrm{det}(G_1,\bff{d}_1(1))>0$.
Without loss of generality we can assume that ${\rm det}(G_1,\bff{d}_1(0))\leq 0$.
Since 
\[
f(t)=\mathrm{det}(G_1,\bff{d}_1(t)) =\mathrm{det}(G_1- v_1,\bff{d}_1')t\bff{d}_v+{\rm det}(G_1,\bff{d}_1(0))
\] 
is a linear function of $t$, then there exists a rational number $0\leq t^* < 1$ such that $\mathrm{det}(G_1,\bff{d}_1(t^*))=0$,
and therefore $\mathrm{det}(G_2,\bff{d}_2(t^*))=0$.

Finally, since $L(G,\bff{d})$ has rank $|V|-1$, the rest follows because $l(\bff{r}_1,\bff{r}_2)$ is a positive integral vector if and only if $\bff{r}_1,\bff{r}_2$
are positive integral vectors. 
\end{proof}

Since $\mathcal{A}(G) \subseteq \mathcal{A}_v(G)$, then by Theorem~\ref{1connected}
\[
\mathcal{A}(G_1)+_v\mathcal{A}(G_2)\subseteq \mathcal{A}(G).
\]
This relation allows having a good approximation of the arithmetical structures on a connected graph 
in terms of the arithmetical structures on their $2$-connected components.
In general, Theorem~\ref{1connected} allows decomposing the problem of finding the arithmetical structures on a connected graph 
into the problem of finding the rational arithmetical structures on its $2$-connected components.
In other words, we transform a problem with a global condition into several local conditions.
However, is complex to describe these generalized arithmetical structures.

The next example illustrates that the equality in $\mathcal{A}(G_1)+_v\mathcal{A}(G_2)\subseteq \mathcal{A}(G)$ does not necessarily hold.
\begin{Example}
Consider the graph $G$ given in Figure~\ref{figura3}, obtained by identifying two of the three vertices of a two complete graphs. 
That is, $G$ has a cut vertex and two complete graphs with three vertices as blocks.
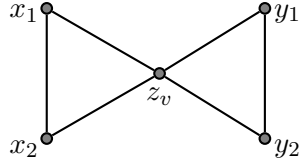
\begin{figure}[h]
\begin{tikzpicture}[scale=1,line width=0.8pt]
\tikzstyle{every node}=[minimum width=4pt, inner sep=0pt, circle]
\draw (0:1) node (v1) [draw,fill=gray] {};
\draw (120:1) node (v2) [draw,fill=gray] {};
\draw (240:1) node (v3) [draw,fill=gray] {};
\draw (120:1)+(2.9,0) node (v4) [draw,fill=gray] {};
\draw (240:1)+(2.9,0) node (v5) [draw,fill=gray] {};
\draw (v1) -- (v2) -- (v3) -- (v1);
\draw (v1) -- (v4) -- (v5) -- (v1);
\draw (1,-0.3) node () {$z_v$};
\draw (-0.8,0.8) node () {$x_1$};
\draw (-0.8,-1) node () {$x_2$};
\draw (2.7,0.8) node () {$y_1$};
\draw (2.7,-1) node () {$y_2$};
\end{tikzpicture}
\caption{The wedge of two complete graphs with three vertices.}
\label{figura3}
\end{figure}

Its generalized Laplacian matrix is given by
\[
L(G,X)=
\left(\begin{array}{ccccc}
x_1& -1 & -1 & 0 &0 \\
-1& x_2 & -1 & 0 &0 \\
-1& -1 & z_v & -1 &-1 \\
0& 0 & -1 & y_1 &-1 \\
0& 0 & -1 & -1 &y_2 \\
\end{array}\right)
\]
It is not difficult to check that the vectors $\bff{d}=(2,3,2,3,7)$ and $\bff{r}=(4,3,5,2,1)$ form an arithmetical structure for $G$.
However this arithmetical structure is not in $\mathcal{A}(K_3)+_v\mathcal{A}(K_3)$.
The arithmetical structure given by $(\bff{d},\bff{r})$ comes from the rational arithmetical structures on $K_3$
\[
\big((2,3,7/5),(4,3,5)\big) \text{ and }\big(3/5,3,7),(5,2,1)\big).
\]
\end{Example}

The rational arithmetical structures on a graph are as important as their arithmetical structures.
Moreover, any rational arithmetical structure can be extended to an arithmetical structure on some induced supergraph.

\begin{Theorem}\label{extension}
Let $G$ be a graph and $U\subseteq V(G)$.
If $(\bff{d},\bff{r})$ is a $U$-rational arithmetical structure on $G$, then there exists an induced supergraph $H$ of $G$ and
an arithmetical structure $(\bff{d}',\bff{r}')$ of $H$ such that
\[
\bff{d}_w=\bff{d}_w'
\]
for all $w\in V(G)\setminus U$.
\end{Theorem}
\begin{proof}
Let $u\in U$ and $q=1-(\bff{d}_u-\lfloor \bff{d}_u\rfloor)$.
In 1202 in the book \textit{Liber Abaci}, Fibonacci discovered an algorithm that shows that every ordinary fraction has an Egyptian Fraction form, see for instance~\cite{sylvester} for a proof of its correctness.
That is, there exist $m\in \mathbb{N}_+$ and $\{a_1,\ldots, a_m\}\subset \mathbb{N}_+$ such that
\[
q=\sum_{i=1}^m \frac{1}{a_i}.
\]
Now, consider the star graph $S_m$ with center in $v$ and with $m$ leaves.
It is not difficult to see that $\bff{f}=(1,a_1,\ldots, a_m)$ and $\bff{s}=(c,\frac{c}{a_1},\ldots, \frac{c}{a_m})$, where $c=\mathrm{lcm}\{a_1,\ldots,a_m\}$
is a $v$-rational arithmetical structure on $S_m$.
\[
\left(\begin{array}{cccc}
q& -1 & \cdots & -1 \\
-1& a_1 & 0 & 0 \\
\vdots& 0 & \ddots & 0 \\
-1& 0 & 0 & a_m
\end{array}\right)
\left(\begin{array}{c}
c \\
\frac{c}{a_1} \\
\vdots \\
\frac{c}{a_m}
\end{array}\right)
=\bff{0}
\]
Let $H_u=w(G, u; S_m,v)$.
Clearly $H_u$ is an induced supergraph of $G$ and by Theorem~\ref{1connected}, 
the vectors $\bff{d}'_u=\bff{d}+_v\bff{f}$ and $\bff{r}_u'=l(\bff{r},\bff{s})$
form a $(U\!-\!u)$-rational arithmetical structure on $H_u$.
By successive applications of this procedure we can get a supergraph $H$ of $G$ 
and an arithmetical structure which satisfies the conditions that we want.
\end{proof}

We say that the arithmetical structure $(\bff{d}',\bff{r}')$ of $H$ is an extension of the $U$-rational arithmetical structure $(\bff{d},\bff{r})$ on $G$.
In~\cite[Section 4]{2000-1-Lorenzini} there is given a description of the $l$-part of the group of components of a graph $G$ with a cut vertex $v$ 
in terms of some supergraphs $\tilde{G}_1,\ldots,\tilde{G}_s$ of the connected components of $G-v$.
More precisely, if $G_1,\ldots,G_s$ are subgraphs of $G$ such that $G=\vee_v G_i$, then the $\tilde{G}_1,\ldots,\tilde{G}_s$ 
are graphs obtained from the $G_i$'s by adding some path that begins in $v$.
This corresponds in some sense to an extension (similar to that given in Theorem~\ref{extension}, but using paths instead of star graphs) of the $v$-rational arithmetical structures on the $2$-connected components of $G$ in order to get an arithmetical structure on the $\tilde{G}_i$'s.

\begin{Example}
Consider the rational arithmetical structure (see Figure~\ref{ejemextension}.$(a)$) of the cycle 
with four vertices given by $\bff{d}=(\frac{1}{3},6,\frac{5}{3},9)$ and $\bff{r}=(15,3,3,2)$. 

\begin{figure}[h]
\begin{center}
\begin{tabular}{c@{\hskip 15mm}c@{\hskip 15mm}c}
\begin{tikzpicture}[line width=0.5pt, scale=0.95]
\tikzstyle{every node}=[inner sep=0pt, minimum width=14 pt] 
\draw (0:1) node (v1) [draw, circle] {};
\draw (90:1) node (v2) [draw, circle] {};
\draw (180:1) node (v3) [draw, circle] {};
\draw (270:1) node (v4) [draw, circle] {};
\draw (v1) -- (v2) -- (v3) -- (v4) -- (v1);
\draw (v1) node {\small $\frac{1}{3}$};
\draw (v2) node {$6$};
\draw (v3) node {\small $\frac{5}{3}$};
\draw (v4) node {$9$};
\end{tikzpicture}
&
\begin{tikzpicture}[line width=0.5pt, scale=0.95]
\tikzstyle{every node}=[inner sep=0pt, minimum width=14 pt] 
\draw (0:1) node (v1) [draw, circle] {};
\draw (90:1) node (v2) [draw, circle] {};
\draw (180:1) node (v3) [draw, circle] {};
\draw (270:1) node (v4) [draw, circle] {};
\draw (180:2) node (v5) [draw, circle] {};
\draw (18:2.2) node (v6) [draw, circle] {};
\draw (-18:2.2) node (v7) [draw, circle] {};
\draw (v1) -- (v2) -- (v3) -- (v4) -- (v1);
\draw (v3) -- (v5);
\draw (v1) -- (v6);
\draw (v1) -- (v7);
\draw (v1) node {$1$};
\draw (v2) node {$6$};
\draw (v3) node {$2$};
\draw (v4) node {$9$};
\draw (v5) node {$3$};
\draw (v6) node {$3$};
\draw (v7) node {$3$};
\end{tikzpicture}
&
\begin{tikzpicture}[line width=0.5pt, scale=0.95]
\tikzstyle{every node}=[inner sep=0pt, minimum width=14 pt] 
\draw (0:1) node (v1) [draw, circle] {};
\draw (90:1) node (v2) [draw, circle] {};
\draw (180:1) node (v3) [draw, circle] {};
\draw (270:1) node (v4) [draw, circle] {};
\draw (180:2) node (v5) [draw, circle] {};
\draw (0:2) node (v6) [draw, circle] {};
\draw (2,-1) node (v7) [draw, circle] {};
\draw (v1) -- (v2) -- (v3) -- (v4) -- (v1);
\draw (v3) -- (v5);
\draw (v1) -- (v6);
\draw (v6) -- (v7);
\draw (v1) node {$1$};
\draw (v2) node {$6$};
\draw (v3) node {$2$};
\draw (v4) node {$9$};
\draw (v5) node {$3$};
\draw (v6) node {$2$};
\draw (v7) node {$2$};
\end{tikzpicture} 
\end{tabular}
\end{center}
\caption{$(a)$ Rational arithmetical structures on $C_4$. $(b)$ Supergraph of $C_4$. $(c)$ Supergraph of $C_4$.}\label{ejemextension}
\end{figure}
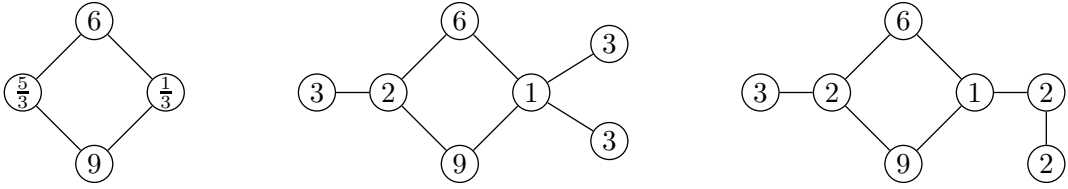
In Figure~\ref{ejemextension}.$(b)$ and~\ref{ejemextension}.$(c)$ we present two extensions of the rational 
arithmetical structure on $C_4$, one using star graphs and the other using paths as a way to construct the induced supergraph.
The first one is given by $\bff{d}=(3,3,1,6,2,9,3)$ and $\bff{r}=(5,5,15,3,3,2,1)$ 
and the second one is given by $\bff{d}=(2,2,1,6,2,9,3)$ and $\bff{r}=(5,10,15,3,3,2,1)$.
\end{Example}

One of the most natural families of graphs with cut vertices are the trees, therefore in the next section we deal with this particular case.

\section{Arithmetical structures on a tree}\label{ontrees}
In this section we study a particular case of graphs with connectivity one: the trees.
We give a description of the arithmetical structures on a star graph and give a partial description of the arithmetical structures on a general tree in terms of its star graphs.

The simplest tree consists of a simple edge, let us say $P_2=uv$. 
It is not difficult to prove that 
\[
\mathcal{A}(P_2)=\{\big((1,1),(1,1)\big)\}.
\]
Furthermore, it is not hard to prove that the $\{u\}$-rational arithmetical structures on $P_2$ are equal to
\[
\mathcal{A}_{u}(P_2)=\{\big({\bf d},{\bf r}\big)\in \mathbb{Q}_+^2\times\mathbb{N}_+^2\,\left| \, {\bf d}=(\frac{1}{d},d), \bff{r}=(d,1) \text{ for some } d\in \mathbb{N}_+\right.\}.
\]
Let $T$ be any tree with $V$ as its vertex set.
Since the blocks of $T$ are its edges, applying Theorem~\ref{1connected} we get that the arithmetical structures on $T$ must satisfy the following equations:

Let $\overset{\rightarrow}{T}$ be the digraph obtained from $T$ by replacing its edges by arcs in both directions and 
\[
a: E(\overset{\rightarrow}{T})\rightarrow \mathbb{Q}_+
\] 
be a weigh function on the arcs of $\overset{\rightarrow}{T}$.
\begin{Proposition}\label{tree}
If $T$ is a tree, then
\[
\mathcal{A}(T)=\left\{\big({\bf d},{\bf r}\big)\in \mathbb{N}_+^{|V(T)|}\times \mathbb{N}_+^{|V(T)|}\, \left|\, \sum_{uv\in E(T)} \bff{a}_{u,v}=\bff{d}_v \,\, \forall \, v\in V(T) \text{ and }
\bff{a}_{u,v}\bff{a}_{v,u}= 1 \, \,\forall \,  uv\in E(T)\right.\right\}. 
\]
\end{Proposition}
\begin{proof}
This follows by Theorem~\ref{1connected} and the previous description of the arithmetical structures on an edge.
\end{proof}

Unfortunately, in general it is not easy to handle this description.
Decomposing a tree into some of its star graphs gives us a better way to try to find a description of the arithmetical structures on $T$.
The next result give us a description of the arithmetical structures on a star graph.

\begin{Corollary}\label{arithStar}
If $S_m$ is the star graph with center at $v$ and $1,\ldots,m$ are its leaves, then
\[
\mathcal{A}(S_m)=\left\{\left({\bf d},{\bf r}\right) \in \mathbb{N}_+^{m+1}\times \mathbb{N}_+^{m+1}\, \left| \, 
\bff{d}_v=\sum_{i=1}^m \frac{1}{\bff{d}_i},\ \bff{r}_v=\mathrm{lcm}\{\bff{d}_i\}_{i=1}^m\textrm{ and }{\bf r}_i=\frac{{\bf r}_v}{\bff{d}_i}\right.\right\}.
\]
\end{Corollary}
\begin{proof}
($\supseteq$)
If we have $\left({\bf d},{\bf r}\right) \in \mathbb{N}_+^{m+1}\times \mathbb{N}_+^{m+1}$ such that
$\bff{d}_v=\sum_{i=1}^m \frac{1}{\bff{d}_i}$ and ${\bf r}_i=\frac{{\bf r}_v}{\bff{d}_i}$ where $\bff{r}_v=\mathrm{lcm}\{\bff{d}_i\}_{i=1}^m$, then
\[
\left(\begin{array}{cccc}
\bff{d}_v& -1 & \cdots & -1 \\
-1& \bff{d}_1 & 0 & 0 \\
\vdots& 0 & \ddots & 0 \\
-1& 0 & 0 & \bff{d}_m
\end{array}\right)
\left(\begin{array}{c}
{\bf r}_v \\
\frac{{\bf r}_v}{\bff{d}_1} \\
\vdots \\
\frac{{\bf r}_v}{\bff{d}_m}
\end{array}\right)
=\bff{0},
\]
and therefore $\left({\bf d},{\bf r}\right)\in \mathcal{A}(S_m)$.

($\subseteq$)
By Proposition~\ref{tree}, since all the vertices of $S_m$, except its center, have degree one, 
$\mathrm{det}(L(S_m, \bff{d}))=0$ if and only if $\bff{d}_v=\sum_{i=1}^m \frac{1}{\bff{d}_i}$ for some $\bff{d}_i \in \mathbb{N}_+$.
Since $L(S_m, \bff{d})$ is an almost non-singular $M$-matrix, the rest follows by checking that $\bff{r}$ is in the kernel of $L(S_m, \bff{d})$.
\end{proof}

The next example shows how to use a decomposition of a tree into star graphs in order to get its arithmetical structures.

\begin{Example}
Let $T$ be a tree decomposed into a star graph with three leaves $S_3$ and a star graph with two leaves $S_2$, as in Figure~\ref{arbol}.$(a)$.
Figures~\ref{arbol} present two arithmetical structures obtained from the decomposition of $T$ into star graphs.
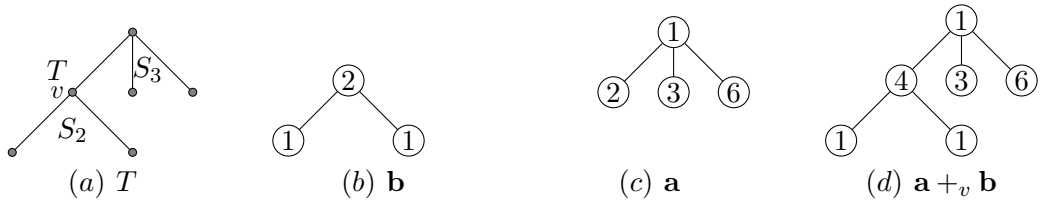
\begin{figure}[h]
\begin{tabular}{c@{\extracolsep{10mm}}c@{\extracolsep{10mm}}c@{\extracolsep{10mm}}c}
\begin{tikzpicture}
[level distance=8mm,
level 1/.style={sibling distance=8mm},
level 2/.style={sibling distance=8mm}]
\tikzstyle{every node}=[minimum width=3pt, inner sep=1pt, circle, draw, fill=gray] 
\node (root) {}
child{
	child
	child[missing]
	child
	}
child	
child;
\node at (root-1) {};
\node at (root-2) {};
\node at (root-3) {};
\node at (root-1-1) {};	
\node at (root-1-3) {};
\tikzstyle{every node}=[minimum width=0pt, inner sep=0.5pt, circle]
\draw (-1,-0.8) node {\small $v$};	
\draw (0.2,-0.5) node {$S_3$};
\draw (-1,-0.5) node {$T$};
\draw (-0.8,-1.3) node {$S_2$};		
\end{tikzpicture}
&
\begin{tikzpicture}
[level distance=8mm,
level 1/.style={sibling distance=8mm},
level 2/.style={sibling distance=8mm}]
\tikzstyle{every node}=[color=white,minimum width=3pt, inner sep=1pt, circle, draw, fill=white] 
\node (root) {}
child[color=white]{
	child[color=black]
	child[missing]
	child[color=black]
	}
child[missing]	
child[missing];
\node at (root-2) {};
\node at (root-3) {};
\tikzstyle{every node}=[minimum width=3pt, inner sep=1pt, circle, draw, fill=white] 
\node at (root-1) {$2$};
\node at (root-1-1) {$1$};	
\node at (root-1-3) {$1$};		
\end{tikzpicture}
&
\begin{tikzpicture}
[level distance=8mm,
level 1/.style={sibling distance=8mm},
level 2/.style={sibling distance=8mm}]
\tikzstyle{every node}=[minimum width=3pt, inner sep=1pt, circle, draw, fill=white] 
\node (root) {$1$}
child{
	child[missing]
	child[missing]
	child[missing]
	}
child	
child;
\node at (root-1) {$2$};
\node at (root-2) {$3$};
\node at (root-3) {$6$};
\tikzstyle{every node}=[color=white,minimum width=3pt, inner sep=1pt, circle, draw, fill=white] 
\node at (root-1-1) {};	
\node at (root-1-3) {};	
\end{tikzpicture}
&
\begin{tikzpicture}
[level distance=8mm,
level 1/.style={sibling distance=8mm},
level 2/.style={sibling distance=8mm}]
\tikzstyle{every node}=[minimum width=3pt, inner sep=1pt, circle, draw, fill=white] 
\node (root) {$1$}
child{
	child
	child[missing]
	child
	}
child	
child;
\node at (root-1) {$4$};
\node at (root-2) {$3$};
\node at (root-3) {$6$};
\node at (root-1-1) {$1$};	
\node at (root-1-3) {$1$};	
\end{tikzpicture}
\\
$(a)$ $T$& $(b)$ ${\bf b}$& $(c)$ ${\bf a}$ & $(d)$ ${\bf a}+_v{\bf b}$
\end{tabular}
\caption{ $(a)$ A tree $T$, which can be decomposed into the star graphs $S_2$ and $S_3$. $(b)$ An arithmetical structure on the star graph $S_2$. $(c)$ An arithmetical structure on the star graph $S_3$. $(d)$ The arithmetical structure on $T$ obtained by composing the arithmetical structures on the star graphs $S_2$ and $S_3$.}
\label{arbol}
\end{figure}

Now, Figure~\ref{descomposicion} shows how the rational arithmetical structures on the two star graphs of $T$ can be composed.
\begin{figure}[h]
\begin{tabular}{c@{\extracolsep{10mm}}c@{\extracolsep{10mm}}c}
\begin{tikzpicture}
[level distance=8mm,
level 1/.style={sibling distance=8mm},
level 2/.style={sibling distance=8mm}]
\tikzstyle{every node}=[color=white,minimum width=3pt, inner sep=1pt, circle, draw, fill=white] 
\node (root) {}
child[color=white]{
	child[color=black]
	child[missing]
	child[color=black]
	}
child[missing]	
child[missing];
\node at (root-2) {};
\node at (root-3) {};
\tikzstyle{every node}=[minimum width=3pt, inner sep=1pt, circle, draw, fill=white] 
\node at (root-1) {\small $\frac{1}{3}$};
\node at (root-1-1) {$6$};	
\node at (root-1-3) {$6$};		
\end{tikzpicture}
&
\begin{tikzpicture}
[level distance=8mm,
level 1/.style={sibling distance=8mm},
level 2/.style={sibling distance=8mm}]
\tikzstyle{every node}=[minimum width=3pt, inner sep=1pt, circle, draw, fill=white] 
\node (root) {$3$}
child{
	child[missing]
	child[missing]
	child[missing]
	}
child	
child;
\node at (root-1) {\small $\frac{2}{3}$};
\node at (root-2) {$1$};
\node at (root-3) {$2$};
\tikzstyle{every node}=[color=white,minimum width=3pt, inner sep=1pt, circle, draw, fill=white] 
\node at (root-1-1) {};	
\node at (root-1-3) {};	
\end{tikzpicture}
&
\begin{tikzpicture}
[level distance=8mm,
level 1/.style={sibling distance=8mm},
level 2/.style={sibling distance=8mm}]
\tikzstyle{every node}=[minimum width=3pt, inner sep=1pt, circle, draw, fill=white] 
\node (root) {$3$}
child{
	child
	child[missing]
	child
	}
child	
child;
\node at (root-1) {$1$};
\node at (root-2) {$1$};
\node at (root-3) {$2$};
\node at (root-1-1) {$6$};	
\node at (root-1-3) {$6$};	
\end{tikzpicture}
\\
$(a)$ ${\bf b}$& $(b)$ ${\bf a}$ & $(c)$ ${\bf a}+_v{\bf b}$
\end{tabular}
\caption{An arithmetical structure on $T$ obtained from two rational arithmetical structures on star graphs.}
\label{descomposicion}
\end{figure}
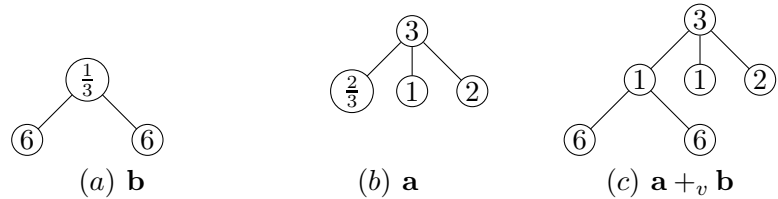
\end{Example}

In the case of trees it is very easy to calculate the critical group associated to any of its arithmetical structures.
Lorenzini in~\cite[Corollary 2.5]{Lorenzini89} proved the next result about the critical group of an arithmetical tree. 

\begin{Proposition}[Corollary 2.5]\label{orderKtree}
Let $T$ be a (simple) tree. If $({\bf d},{\bf r})$ is an arithmetical structure on $T$, then
\[
|K(T,{\bf d})|=\prod_{v\in V(T)} {\bf r}_v^{{\bf deg}_v-2}
\]
where ${\bf deg}$ is the degree vector of $T$.
\end{Proposition}

\end{document}